\documentclass[a4paper,12pt]{amsart}

\usepackage{a4wide}
\usepackage{tikz}

\newcommand{\T}{\mathcal{T}}

\newcommand{\PP}{\mathbb{P}}

\newcommand{\EE}{\mathbb{E}}

\newif\ifdetails
\detailstrue
\newcommand{\DETAIL}[1]%
{\ifdetails\par\fbox{\begin{minipage}{0.9\linewidth}\textit{Detail:}
      #1\end{minipage}}\par\fi}
\newcommand{\TODO}[1]%
{\ifdetails\par\fbox{\begin{minipage}{0.9\linewidth}\textbf{TODO:}
      #1\end{minipage}}\par\fi}

\newtheorem{lemma}{Lemma}
\newtheorem{prop}[lemma]{Proposition}
\newtheorem{theorem}[lemma]{Theorem}

\theoremstyle{remark}

\newtheorem{remark}{Remark}
\newtheorem{conjecture}{Conjecture}
\newtheorem{question}{Question}

\newcommand{\CB}{\operatorname{CB}}
\newcommand{\C}{\operatorname{C}}
\newcommand{\E}{\operatorname{E}}
\newcommand{\Crt}{\operatorname{crt}}
\newcommand{\Cr}{\operatorname{cr}}

\newcommand{\old}[1]{{}}

\title{Inducibility in binary trees and crossings in random tanglegrams}

\author{\'Eva Czabarka, L\'aszl\'o A. Sz\'ekely }
\address{\'Eva Czabarka and L\'aszl\'o A. Sz\'ekely\\ Department of Mathematics \\ University of South Carolina \\ Columbia, SC 29208 \\ USA}
\email{\{czabarka,szekely\}@math.sc.edu }
\thanks{The second author was supported in part by the  NSF DMS,  grant number 1300547,  the third author was supported by the National
Research Foundation of South Africa, grant number 96236.}
\author{Stephan Wagner}
\address{Stephan Wagner\\ Department of Mathematical Sciences \\ Stellenbosch University \\ Private Bag X1, Matieland 7602 \\ South Africa}
\email{swagner@sun.ac.za}

\begin{document}

\begin{abstract}
In analogy to other concepts of a similar nature, we define the inducibility of a rooted binary tree. Given a fixed rooted binary tree $B$ with $k$ leaves, we let $\gamma(B,T)$ be the proportion of all subsets of $k$ leaves in $T$ that induce a tree isomorphic to $B$. The inducibility of $B$ is $\limsup_{|T| \to \infty} \gamma(B,T)$. We determine the inducibility in some special cases, show that every binary tree has positive inducibility and prove that caterpillars are the only binary trees with inducibility $1$. We also formulate some open problems and conjectures on the inducibility. Finally, we present an application to crossing numbers of random tanglegrams.
\end{abstract}

\maketitle

\section{introduction}
The concept of inducibility of graphs goes back to a 1975 paper by Pippenger and Golumbic \cite{Pippenger}. For two graphs $G$ and $H$ with $k$ and $n$ vertices respectively, let $\mathcal{I}(G,H)$ be the number of induced subgraphs of $H$ isomorphic to $G$, and $I(G,H) = \mathcal{I}(G,H)/\binom{n}{k}$ its normalized version, which lies between $0$ and $1$. Furthermore, let $I(G,n) = \max_{|H| = n} I(G,H)$ be the maximum over all $n$-vertex graphs $H$. The limit
$$\lim_{n \to\infty} I(G,n)$$
is called the \emph{inducibility} of $G$. The concept is still investigated vigorously to this day, see \cite{Even-Zohar,Hatami} for some recent results.

Bubeck and Linial \cite{BubeckLinial} recently analyzed the distribution of subtrees of fixed order by isomorphism type, and defined inducibility of trees in this context. A notable difference in the definitions is that Bubeck and Linial normalize by the number of $k$-vertex subtrees ($k$-vertex subsets that induce a tree) rather than the total number of all $k$-vertex subsets.

In analogy to these concepts, we are concerned with extremal problems in rooted binary trees, 
and in particular we define the inducibility of a rooted binary tree.

Let $T$ be a rooted binary tree. Here and in the following, whenever we write ``binary tree'', we refer to rooted binary trees, i.e., every vertex is either a leaf or it has exactly two children. The number of leaves of $T$ is denoted by $|T|$. We consider two binary trees as distinct if and only if they are not isomorphic. Each of the $\binom{|T|}{k}$ choices of $k$ leaves induces another rooted binary tree by taking the smallest subtree containing all $k$ leaves and suppressing vertices of degree $2$, see Figure~\ref{fig:example}. The induced binary tree has a root in the natural way. 

The study of induced binary subtrees of (rooted or unrooted) binary trees is topical in the phylogenetic literature \cite{SempleSteel}.
For $k \leq 3$, there is only one possibility for the shape of the induced tree (up to isomorphism). However, this changes for larger values of $k$. It is well known that the number of rooted binary trees with $k$ leaves up to isomorphism is the Wedderburn-Etherington number $W_k$. The first few of these numbers are
$$1, 1, 1, 2, 3, 6, 11, 23, 46, 98, 207, 451, 983, \ldots$$

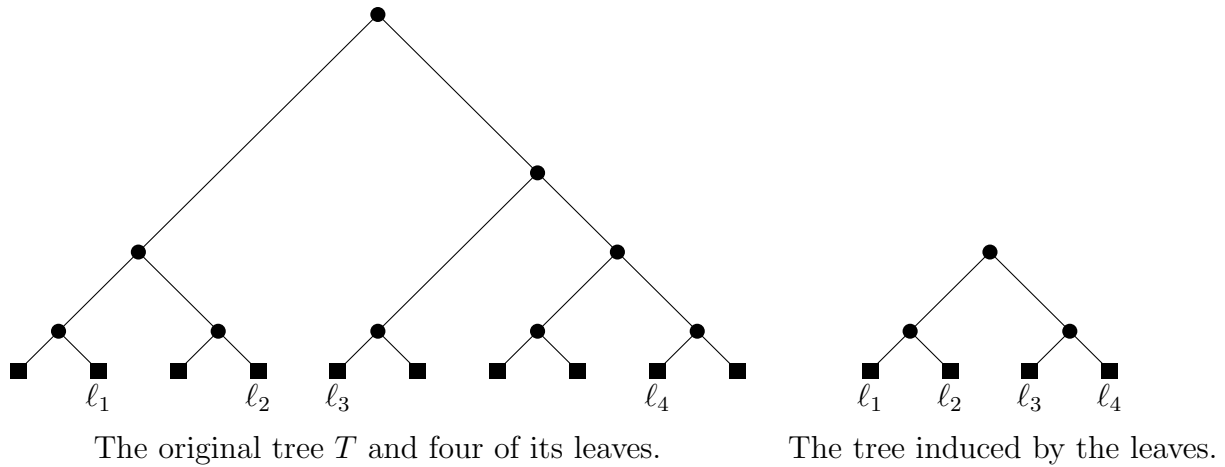
\begin{figure}[htbp]
\begin{center}
\begin{tikzpicture}[scale=0.7]

        \node[fill=black,rectangle,inner sep=3pt]  at (0,0) {};
        \node[fill=black,rectangle,inner sep=3pt]  at (1.5,0) {};
        \node[fill=black,rectangle,inner sep=3pt]  at (3,0) {};
        \node[fill=black,rectangle,inner sep=3pt]  at (4.5,0) {};
        \node[fill=black,rectangle,inner sep=3pt]  at (6,0) {};
        \node[fill=black,rectangle,inner sep=3pt]  at (7.5,0) {};
        \node[fill=black,rectangle,inner sep=3pt]  at (9,0) {};
        \node[fill=black,rectangle,inner sep=3pt]  at (10.5,0) {};
        \node[fill=black,rectangle,inner sep=3pt]  at (12,0) {};
        \node[fill=black,rectangle,inner sep=3pt]  at (13.5,0) {};

        \node[fill=black,circle,inner sep=2pt]  at (6.75,6.75) {};
        \node[fill=black,circle,inner sep=2pt]  at (2.25,2.25) {};
        \node[fill=black,circle,inner sep=2pt]  at (9.75,3.75) {};
        \node[fill=black,circle,inner sep=2pt]  at (11.25,2.25) {};
        \node[fill=black,circle,inner sep=2pt]  at (0.75,0.75) {};
        \node[fill=black,circle,inner sep=2pt]  at (3.75,0.75) {};
        \node[fill=black,circle,inner sep=2pt]  at (6.75,0.75) {};
        \node[fill=black,circle,inner sep=2pt]  at (9.75,0.75) {};
        \node[fill=black,circle,inner sep=2pt]  at (12.75,0.75) {};

	\draw (0,0)--(6.75,6.75)--(13.5,0);
	\draw (4.5,0)--(2.25,2.25);
	\draw (6,0)--(9.75,3.75);
	\draw (9,0)--(11.25,2.25);
	\draw (1.5,0)--(0.75,0.75);
	\draw (3,0)--(3.75,0.75);
	\draw (7.5,0)--(6.75,0.75);
	\draw (10.5,0)--(9.75,0.75);
	\draw (12,0)--(12.75,0.75);

	\node at (1.5,-0.5) {$\ell_1$};
	\node at (4.5,-0.5) {$\ell_2$};
	\node at (6,-0.5) {$\ell_3$};
	\node at (12,-0.5) {$\ell_4$};

	\node at (6.75,-1.5) {The original tree $T$ and four of its leaves.};

        \node[fill=black,rectangle,inner sep=3pt]  at (16,0) {};
        \node[fill=black,rectangle,inner sep=3pt]  at (17.5,0) {};
        \node[fill=black,rectangle,inner sep=3pt]  at (19,0) {};
        \node[fill=black,rectangle,inner sep=3pt]  at (20.5,0) {};

        \node[fill=black,circle,inner sep=2pt]  at (16.75,0.75) {};
        \node[fill=black,circle,inner sep=2pt]  at (18.25,2.25) {};
        \node[fill=black,circle,inner sep=2pt]  at (19.75,0.75) {};

	\draw (16,0)--(18.25,2.25)--(20.5,0);
	\draw (16.75,0.75)--(17.5,0);
	\draw (19.75,0.75)--(19,0);

	\node at (16,-0.5) {$\ell_1$};
	\node at (17.5,-0.5) {$\ell_2$};
	\node at (19,-0.5) {$\ell_3$};
	\node at (20.5,-0.5) {$\ell_4$};

	\node at (18.5,-1.5) {The tree induced by the leaves.};
\end{tikzpicture}
\end{center}
\caption{A rooted binary tree and the tree induced by four of its leaves.}\label{fig:example}
\end{figure}

For a fixed binary tree $B$ with $k$ leaves, let us write $c(B,T)$ for the number of sets of $k$ leaves in $T$ that induce a tree isomorphic to $B$. We will be specifically interested in the quotient
$$\gamma(B,T) = \frac{c(B,T)}{\binom{|T|}{k}}$$
and the maximum of this quantity in the limit, the \emph{inducibility}
$$i(B) = \limsup_{|T| \to \infty} \gamma(B,T).$$
We remark that the inducibility of a binary tree is a number in the interval $[0,1]$ by definition. In the following session, we determine the exact value of the inducibility for two special classes of binary trees (caterpillars and ``even'' trees, which are a generalization of complete binary trees). We also prove that every binary tree has positive inducibility and that caterpillars are the only binary trees with inducibility $1$, see Section~\ref{sec:general}. In our last section we use our results to establish the order of magnitude of the crossing number of a random tanglegram. The introduction to tanglegrams and their
crossing numbers is postponed to Section~\ref{tangle}.

\section{Special cases}

The simplest case to consider is the \emph{caterpillar}, a binary tree whose internal vertices form a path (see Figure~\ref{fig:cater}). Let us denote the caterpillar with $k$ leaves by $\C_k$.

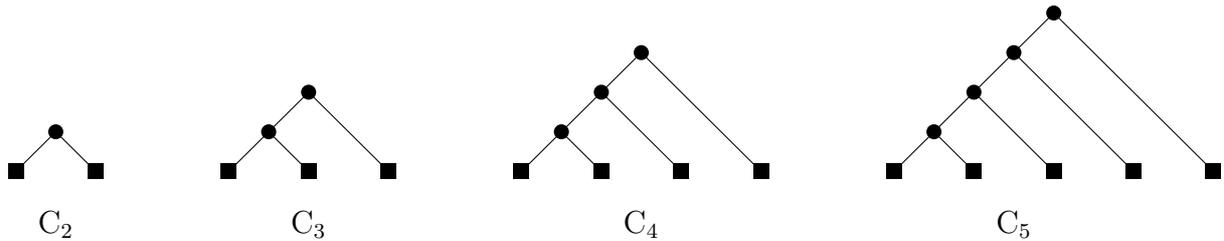
\begin{figure}[htbp]
\begin{center}
\begin{tikzpicture}[scale=0.7]

        \node[fill=black,rectangle,inner sep=3pt]  at (0,0) {};
        \node[fill=black,rectangle,inner sep=3pt]  at (1.5,0) {};

        \node[fill=black,circle,inner sep=2pt]  at (0.75,0.75) {};

	\draw (0,0)--(0.75,0.75)--(1.5,0);

	\node at (0.75,-1) {$\C_2$};

        \node[fill=black,rectangle,inner sep=3pt]  at (4,0) {};
        \node[fill=black,rectangle,inner sep=3pt]  at (5.5,0) {};
        \node[fill=black,rectangle,inner sep=3pt]  at (7,0) {};

        \node[fill=black,circle,inner sep=2pt]  at (4.75,0.75) {};
        \node[fill=black,circle,inner sep=2pt]  at (5.5,1.5) {};

	\draw (4,0)--(5.5,1.5)--(7,0);
	\draw (4.75,0.75)--(5.5,0);

	\node at (5.5,-1) {$\C_3$};

        \node[fill=black,rectangle,inner sep=3pt]  at (9.5,0) {};
        \node[fill=black,rectangle,inner sep=3pt]  at (11,0) {};
        \node[fill=black,rectangle,inner sep=3pt]  at (12.5,0) {};
        \node[fill=black,rectangle,inner sep=3pt]  at (14,0) {};

        \node[fill=black,circle,inner sep=2pt]  at (10.25,0.75) {};
        \node[fill=black,circle,inner sep=2pt]  at (11,1.5) {};
        \node[fill=black,circle,inner sep=2pt]  at (11.75,2.25) {};

	\draw (9.5,0)--(11.75,2.25)--(14,0);
	\draw (10.25,0.75)--(11,0);
	\draw (11,1.5)--(12.5,0);

	\node at (11.75,-1) {$\C_4$};

        \node[fill=black,rectangle,inner sep=3pt]  at (16.5,0) {};
        \node[fill=black,rectangle,inner sep=3pt]  at (18,0) {};
        \node[fill=black,rectangle,inner sep=3pt]  at (19.5,0) {};
        \node[fill=black,rectangle,inner sep=3pt]  at (21,0) {};
        \node[fill=black,rectangle,inner sep=3pt]  at (22.5,0) {};

        \node[fill=black,circle,inner sep=2pt]  at (17.25,0.75) {};
        \node[fill=black,circle,inner sep=2pt]  at (18,1.5) {};
        \node[fill=black,circle,inner sep=2pt]  at (18.75,2.25) {};
        \node[fill=black,circle,inner sep=2pt]  at (19.5,3) {};

	\draw (16.5,0)--(19.5,3)--(22.5,0);
	\draw (17.25,0.75)--(18,0);
	\draw (18,1.5)--(19.5,0);
	\draw (18.75,2.25)--(21,0);

	\node at (18.75,-1) {$\C_5$};
\end{tikzpicture}
\end{center}
\caption{Caterpillars.}\label{fig:cater}
\end{figure}

\begin{theorem}\label{thm:cater}
The inducibility of the $k$-leaf caterpillar $\C_k$ is $1$ for every $k$.
\end{theorem}

\begin{proof}
Simply note that every tree induced by the leaves of a caterpillar is again a caterpillar, thus $\gamma(\C_k,\C_n) = 1$.
\end{proof}

The first nontrivial case occurs for binary trees with four leaves, which is the first case where there are nonisomorphic binary trees with the same number of leaves. As it turns out, the only other binary tree in this case no longer has inducibility $1$. This tree is a special case of a complete binary tree (see Figure~\ref{fig:complete}). We will denote the complete binary tree of height $h$ (with $2^h$ leaves) by $\CB_h$.

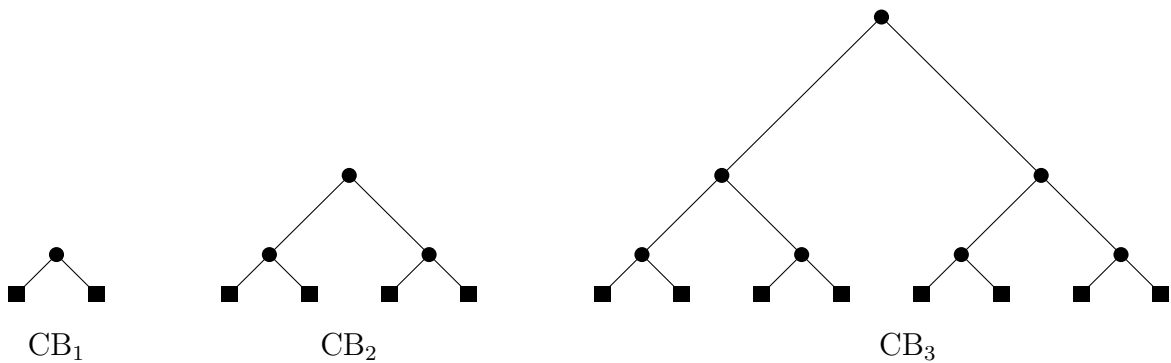
\begin{figure}[htbp]
\begin{center}
\begin{tikzpicture}[scale=0.7]

        \node[fill=black,rectangle,inner sep=3pt]  at (0,0) {};
        \node[fill=black,rectangle,inner sep=3pt]  at (1.5,0) {};

        \node[fill=black,circle,inner sep=2pt]  at (0.75,0.75) {};

	\draw (0,0)--(0.75,0.75)--(1.5,0);

	\node at (0.75,-1) {$\CB_1$};

        \node[fill=black,rectangle,inner sep=3pt]  at (4,0) {};
        \node[fill=black,rectangle,inner sep=3pt]  at (5.5,0) {};
        \node[fill=black,rectangle,inner sep=3pt]  at (7,0) {};
        \node[fill=black,rectangle,inner sep=3pt]  at (8.5,0) {};

        \node[fill=black,circle,inner sep=2pt]  at (4.75,0.75) {};
        \node[fill=black,circle,inner sep=2pt]  at (6.25,2.25) {};
        \node[fill=black,circle,inner sep=2pt]  at (7.75,0.75) {};

	\draw (4,0)--(6.25,2.25)--(8.5,0);
	\draw (4.75,0.75)--(5.5,0);
	\draw (7.75,0.75)--(7,0);

	\node at (6.25,-1) {$\CB_2$};

        \node[fill=black,rectangle,inner sep=3pt]  at (11,0) {};
        \node[fill=black,rectangle,inner sep=3pt]  at (12.5,0) {};
        \node[fill=black,rectangle,inner sep=3pt]  at (14,0) {};
        \node[fill=black,rectangle,inner sep=3pt]  at (15.5,0) {};
        \node[fill=black,rectangle,inner sep=3pt]  at (17,0) {};
        \node[fill=black,rectangle,inner sep=3pt]  at (18.5,0) {};
        \node[fill=black,rectangle,inner sep=3pt]  at (20,0) {};
        \node[fill=black,rectangle,inner sep=3pt]  at (21.5,0) {};

        \node[fill=black,circle,inner sep=2pt]  at (11.75,0.75) {};
        \node[fill=black,circle,inner sep=2pt]  at (14.75,0.75) {};
        \node[fill=black,circle,inner sep=2pt]  at (17.75,0.75) {};
        \node[fill=black,circle,inner sep=2pt]  at (20.75,0.75) {};
        \node[fill=black,circle,inner sep=2pt]  at (13.25,2.25) {};
        \node[fill=black,circle,inner sep=2pt]  at (19.25,2.25) {};
        \node[fill=black,circle,inner sep=2pt]  at (16.25,5.25) {};

	\draw (11,0)--(16.25,5.25)--(21.5,0);
	\draw (15.5,0)--(13.25,2.25);
	\draw (17,0)--(19.25,2.25);
	\draw (12.5,0)--(11.75,0.75);
	\draw (14,0)--(14.75,0.75);
	\draw (18.5,0)--(17.75,0.75);
	\draw (20,0)--(20.75,0.75);

	\node at (16.75,-1) {$\CB_3$};

\end{tikzpicture}
\end{center}
\caption{Complete binary trees.}\label{fig:complete}
\end{figure}

\begin{prop}\label{prop:cb2}
For a tree $T$ with $n$ leaves, we have
$$c(\CB_2,T) \leq \frac{n(n-1)(n-2)(3n-5)}{168}.$$
Equality holds if and only if $T$ is a complete binary tree. In particular, the inducibility of $\CB_2$ is $\frac37$.
\end{prop}

\begin{proof}
We prove the statement by induction on $n$. For simplicity, let us write
$$P(n) = \frac{n(n-1)(n-2)(3n-5)}{168}$$
for the polynomial bound.
For $n \leq 3$, the inequality is trivial since there clearly cannot be any copies of $\CB_2$ in $T$, while $P(n) = 0$ for $n \leq 2$ and $P(3) = \frac17$. For the induction step, suppose that the two branches $T_1$ and $T_2$ of $T$ have $k$ and $n-k$ leaves respectively. There are three possibilities for a subset of four leaves:
\begin{itemize}
\item All four leaves belong to the same branch: clearly, the total number of these  subsets that induce $\CB_2$ is
$$c(\CB_2,T_1) + c(\CB_2,T_2).$$
\item Three of the leaves belong to one branch, the fourth leaf to the other. Any such set of four leaves induces a caterpillar $\C_4$.
\item Each of the branches contains two of the leaves: in this case, the four leaves always induce $\CB_2$.
\end{itemize}
Combining the three cases, we find that
$$c(\CB_2,T) = c(\CB_2,T_1) + c(\CB_2,T_2) + \binom{k}{2} \binom{n-k}{2}.$$
Now we can invoke the induction hypothesis, which gives us
$$c(\CB_2,T) \leq P(k) + P(n-k) + \binom{k}{2} \binom{n-k}{2}.$$
The derivative of the right side with respect to $k$ is
$$- \frac{(n-2k)(n^2-8kn+8k^2)}{14},$$
so we find that it attains its unique maximum at $k = n/2$ (taking $k$ in the interval $[1,n-1]$, the two roots of the second factor are minima). Consequently,
$$c(\CB_2,T) \leq 2P \Big(\frac{n}2\Big) + \binom{n/2}{2}^2 = P(n),$$
which is what we wanted to prove. By the induction hypothesis, equality can only hold if $k = n/2$ and both branches are complete binary trees. In this case, $T$ is a complete binary tree as well.

The statement on the inducibility follows by taking the limit
$$\lim_{n \to \infty} \frac{P(n)}{\binom{n}{4}} = \frac37.$$
\end{proof}

Along the same lines, we also obtain:

\begin{prop}
The inducibility of the tree $\operatorname{A}^5_1$ shown on the left in Figure~\ref{fig:five} is $\frac23$. This limit is achieved in complete binary trees. 
\end{prop}

Except for $\C_5$ and $\operatorname{A}_1^5$, there is only one more tree with five leaves, shown on the right in Figure~\ref{fig:five}. Determining its inducibility seems quite a bit harder. Numerical experiments indicate that it is close to $\frac14$.

\begin{question}
What is the inducibility $i(\operatorname{A}^5_2)$ of the tree $\operatorname{A}^5_2$ shown on the right in Figure~\ref{fig:five}?
\end{question}

\begin{figure}[htbp]
\begin{center}
\begin{tikzpicture}

        \node[fill=black,rectangle,inner sep=3pt]  at (-3,0) {};
        \node[fill=black,rectangle,inner sep=3pt]  at (-1.5,0) {};
        \node[fill=black,rectangle,inner sep=3pt]  at (0,0) {};
        \node[fill=black,rectangle,inner sep=3pt]  at (1.5,0) {};
        \node[fill=black,rectangle,inner sep=3pt]  at (3,0) {};

        \node[fill=black,circle,inner sep=2pt]  at (-2.25,0.75) {};
        \node[fill=black,circle,inner sep=2pt]  at (-1.5,1.5) {};
        \node[fill=black,circle,inner sep=2pt]  at (0,3) {};
        \node[fill=black,circle,inner sep=2pt]  at (2.25,0.75) {};

	\draw (-3,0)--(0,3)--(3,0);
	\draw (-1.5,0)--(-2.25,0.75);
	\draw (1.5,0)--(2.25,0.75);
	\draw (0,0)--(-1.5,1.5);

	\node at (0,-1) {$\operatorname{A}^5_1$};

        \node[fill=black,rectangle,inner sep=3pt]  at (5,0) {};
        \node[fill=black,rectangle,inner sep=3pt]  at (6.5,0) {};
        \node[fill=black,rectangle,inner sep=3pt]  at (8,0) {};
        \node[fill=black,rectangle,inner sep=3pt]  at (9.5,0) {};
        \node[fill=black,rectangle,inner sep=3pt]  at (11,0) {};

        \node[fill=black,circle,inner sep=2pt]  at (7.25,0.75) {};
        \node[fill=black,circle,inner sep=2pt]  at (8.75,2.25) {};
        \node[fill=black,circle,inner sep=2pt]  at (8,3) {};
        \node[fill=black,circle,inner sep=2pt]  at (10.25,0.75) {};

	\draw (5,0)--(8,3)--(11,0);
	\draw (6.5,0)--(8.75,2.25);
	\draw (8,0)--(7.25,0.75);
	\draw (9.5,0)--(10.25,0.75);

	\node at (8,-1) {$\operatorname{A}^5_2$};

\end{tikzpicture}
\end{center}
\caption{Two binary rooted trees with five leaves.}\label{fig:five}
\end{figure}
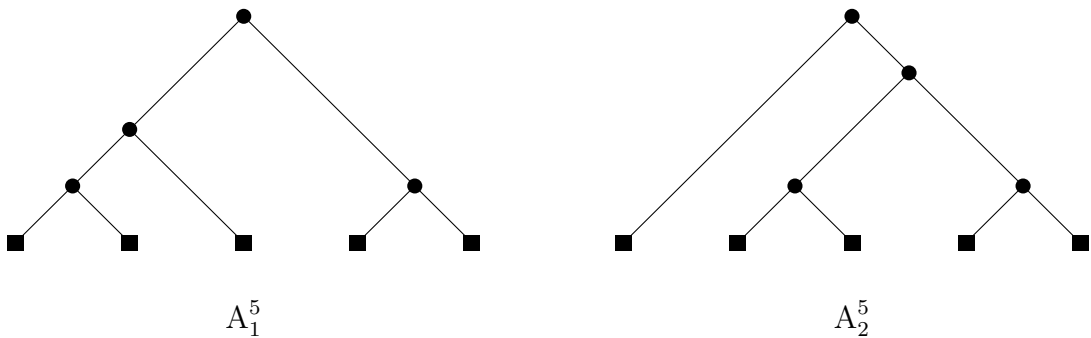

\begin{remark}
It might seem paradoxical that the tree $\operatorname{A}^5_1$, which contains $\CB_2$ as a subtree, has greater inducibility than $\CB_2$. However, this is in essence just a consequence of the fact that a $2$-$3$ split in the binomial distribution on five objects has greater probability than a $2$-$2$ split on four objects.
\end{remark}

Proposition~\ref{prop:cb2} generalizes to complete binary trees of arbitrary height and even to a more general class of binary trees: let us call a binary tree {\em even} if for every internal vertex, the number of leaves in the two subtrees below it differ by at most 1. 
It is easy to see that there is a unique even tree for every given number $n$ of leaves, which we denote by $\E_n$. For example, $\operatorname{A}_5^1=\E_5$ and $\CB_h = \E_{2^h}$, see also Figure~\ref{fig:even} for another example.

\begin{figure}[htbp]
\begin{center}
\begin{tikzpicture}[scale=0.7]

        \node[fill=black,rectangle,inner sep=3pt]  at (0,0) {};
        \node[fill=black,rectangle,inner sep=3pt]  at (1.5,0) {};
        \node[fill=black,rectangle,inner sep=3pt]  at (3,0) {};
        \node[fill=black,rectangle,inner sep=3pt]  at (4.5,0) {};
        \node[fill=black,rectangle,inner sep=3pt]  at (6,0) {};
        \node[fill=black,rectangle,inner sep=3pt]  at (7.5,0) {};
        \node[fill=black,rectangle,inner sep=3pt]  at (9,0) {};

        \node[fill=black,circle,inner sep=2pt]  at (0.75,0.75) {};
        \node[fill=black,circle,inner sep=2pt]  at (1.5,1.5) {};
        \node[fill=black,circle,inner sep=2pt]  at (5.25,0.75) {};
        \node[fill=black,circle,inner sep=2pt]  at (8.25,0.75) {};
        \node[fill=black,circle,inner sep=2pt]  at (4.5,4.5) {};
        \node[fill=black,circle,inner sep=2pt]  at (6.75,2.25) {};

	\draw (0,0)--(4.5,4.5)--(9,0);
	\draw (4.5,0)--(6.75,2.25);
	\draw (3,0)--(1.5,1.5);
	\draw (1.5,0)--(0.75,0.75);
	\draw (6,0)--(5.25,0.75);
	\draw (7.5,0)--(8.25,0.75);

	\node at (4.5,-1) {$\E_7$};
\end{tikzpicture}
\end{center}
\caption{The even tree $\E_7$.}\label{fig:even}
\end{figure}
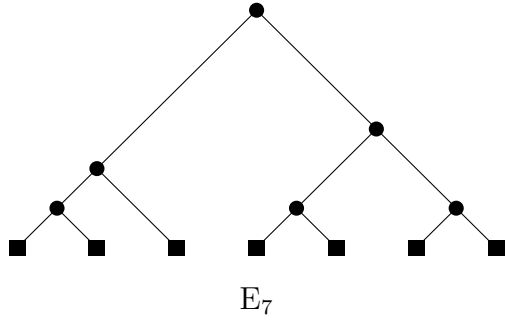

To determine the inducibility of even trees,  we first need a simple lemma.

\begin{lemma}\label{lem:f_max1}
For every positive integer $k\geq 1$, the function
$$f(x) = \frac{x^k(1-x)^k}{1-x^{2k}-(1-x)^{2k}}$$
on the interval $(0,1)$ has its maximum at $x = \frac12$:
$$f(x) \leq f \Big(\frac12\Big) = \frac{1}{2^{2k}-2}$$
for all $x \in (0,1)$.

Likewise, the function
$$g(x) = \frac{x^k(1-x)^{k+1}+x^{k+1}(1-x)^k}{1-x^{2k+1}-(1-x)^{2k+1}}$$
on the interval $(0,1)$ has its maximum at $x = \frac12$:
$$g(x) \leq g \Big(\frac12\Big) = \frac{1}{2^{2k}-1}$$
for all $x \in (0,1)$.

\end{lemma}

\begin{proof}
By the binomial theorem, we have
$$\frac{1}{f(x)} = \frac{\sum_{j=1}^{2k-1} \binom{2k}{j} x^j(1-x)^{2k-j}}{x^k(1-x)^k} = \sum_{j=1}^{2k-1} \binom{2k}{j} x^{j-k}(1-x)^{k-j}.$$
Now we group the terms pairwise (the $j$-th term is paired with the $(2k-j)$-th): since
$$x^{j-k}(1-x)^{k-j} + x^{k-j}(1-x)^{j-k} \geq 2$$
by the inequality between arithmetic and geometric mean, we obtain
$$\frac{1}{f(x)} \geq \sum_{j=1}^{2k-1} \binom{2k}{j} = 2^{2k}-2,$$
which proves the lemma for the function $f(x)$. The function $g(x)$ is treated in a similar way:
$$\frac{1}{g(x)} = \frac{\sum_{j=1}^{2k} \binom{2k+1}{j} x^j(1-x)^{2k+1-j}}{x^k(1-x)^{k+1}+x^{k+1}(1-x)^k},$$
and the inequality
$$x^j(1-x)^{2k+1-j} + x^{2k+1-j}(1-x)^{j} \geq x^k(1-x)^{k+1}+x^{k+1}(1-x)^k$$
is readily seen to be equivalent to
$$\big((1-x)^{k-j}-x^{k-j}\big)\big((1-x)^{k+1-j}-x^{k+1-j}\big) \geq 0,$$
which clearly holds for all $x \in (0,1)$ and all $j \in \{1,2,\ldots,2k\}$. The desired inequality for $g(x)$ follows.
\end{proof}

\begin{theorem}\label{thm:even}
Define the sequence $c_1,c_2,\ldots$ of rational numbers recursively by $c_1 = 1$ and
\begin{align*}
c_{2s} &= \frac{c_s^2}{2^{2s}-2},\\
c_{2s+1} &= \frac{c_s c_{s+1}}{2^{2s}-1}.
\end{align*}
The inducibility of the even tree $\E_r$ is equal to $r! \cdot c_r$ for every $r \geq 1$.
\end{theorem}
\begin{proof}
We prove by double induction on $n$ and $r$ that
$$c(\E_r,T) \leq c_r n^{r}$$
for every binary tree $T$ with $n$ leaves. The case $r=1$ is trivial. Moreover, the inequality certainly also holds for $n < r$. 

For the induction step, consider a binary $T$ with $n$ leaves and let $T_1$ and $T_2$ be the two branches of $T$ as before. The number of leaves of $T_1$ will be denoted by $k$. For even $r$ ($r=2s$), we note that the two branches of $\E_{2s}$ are both isomorphic to $\E_s$, which gives us
$$c(\E_{2s},T) = c(\E_{2s},T_1) + c(\E_{2s},T_2) + c(\E_s,T_1) c(\E_s,T_2).$$
The induction hypothesis with respect to $n$ gives us $c(\E_{2s},T_1) \leq c_{2s} k^{2s}$ as well as $c(\E_{2s},T_2) \leq c_{2s} (n-k)^{2s}$, the induction hypothesis with respect to $r$ yields $c(\E_s,T_1) \leq c_s k^{s}$ and $c(\E_s,T_2) \leq c_s (n-k)^{s}$. Putting them together, we obtain
$$c(\E_{2s},T) \leq c_{2s} \big( k^{2s} + (n-k)^{2s} \big) + c_s^2 k^{s} (n-k)^{s}.$$
Dividing by $n^{2s}$ and setting $x = k/n$ gives us
\begin{align*}
\frac{c(\E_{2s},T)}{n^{2s}} &\leq c_{2s} \big(x^{2s} + (1-x)^{2s} \big) + c_s^2 x^{s} (1-x)^{s} \\
& = c_{2s} \big(x^{2s} + (1-x)^{2s} \big) + (2^{2s}-2) c_{2s} x^{s} (1-x)^{s} \\
&\leq c_{2s},
\end{align*}
where the last step follows from Lemma~\ref{lem:f_max1} by a simple manipulation. This proves the desired inequality. 

The case that $r$ is odd ($r=2s+1$) is treated in the same way: here we have
$$c(\E_{2s+1},T) = c(\E_{2s+1},T_1) + c(\E_{2s+1},T_2) + c(\E_s,T_1) c(\E_{s+1},T_2) + c(\E_{s+1},T_1) c(\E_{s},T_2),$$
and the induction step runs along the same lines (now using the second part of Lemma~\ref{lem:f_max1}).

In the limit, we now obtain
$$i(\E_r) \leq \lim_{n \to \infty} \frac{c_r n^{r}}{\binom{n}{r}} = r! \cdot c_r.$$
The reverse inequality follows by considering $c(\E_r,\E_n)$: we show that
$$\lim_{n \to \infty} \frac{c(\E_r,\E_n)}{n^r} = c_r,$$
again using induction (with respect to $r$). The cases $r=1$ and $r=2$ are trivial, so we focus on the induction step. For $r=2s$, we get
$$c(\E_{2s},\E_n) = c(\E_{2s},\E_{\lfloor n/2 \rfloor}) + c(\E_{2s},\E_{\lceil n/2 \rceil})+ c(\E_s,\E_{\lfloor n/2 \rfloor})c(\E_s,\E_{\lceil n/2 \rceil}).$$
Divide by $n^{2s}$ and take the limit superior:
\begin{align*}
\limsup_{n \to \infty} \frac{c(\E_{2s},\E_n)}{n^{2s}} &\leq 2^{-2s} \limsup_{n \to \infty} \frac{c(\E_{2s},\E_{\lfloor n/2 \rfloor})}{(n/2)^{2s}} +  2^{-2s} \limsup_{n \to \infty} \frac{c(\E_{2s},\E_{\lceil n/2 \rceil})}{(n/2)^{2s}} \\
&\quad + 2^{-2s} \Big( \limsup_{n \to \infty} \frac{c(\E_{s},\E_{\lfloor n/2 \rfloor})}{(n/2)^{s}} \Big) \Big( \limsup_{n \to \infty} \frac{c(\E_{s},\E_{\lceil n/2 \rceil})}{(n/2)^{s}} \Big).
\end{align*}
By the induction hypothesis, this implies
$$\limsup_{n \to \infty} \frac{c(\E_{2s},\E_n)}{n^{2s}} \leq 2^{1-2s} \Big( \limsup_{n \to \infty} \frac{c(\E_{2s},\E_n)}{n^{2s}} \Big) + 2^{-2s} \cdot c_s^2$$
and thus
$$\limsup_{n \to \infty} \frac{c(\E_{2s},\E_n)}{n^{2s}} \leq \frac{c_s^2}{2^{2s}-2} = c_{2s}.$$
The same reasoning also yields
$$\liminf_{n \to \infty} \frac{c(\E_{2s},\E_n)}{n^{2s}} \geq \frac{c_s^2}{2^{2s}-2} = c_{2s},$$
which shows that
$$\lim_{n \to \infty} \frac{c(\E_{2s},\E_n)}{n^{2s}} = \frac{c_s^2}{2^{2s}-2} = c_{2s}.$$
This completes the proof for even $r$, and the identity
$$\lim_{n \to \infty} \frac{c(\E_{2s+1},\E_n)}{n^{2s}} = \frac{c_sc_{s+1}}{2^{2s}-1} = c_{2s+1}$$
is established in the same way.
\end{proof}

We conjecture that a finite analogue of Theorem~\ref{thm:even} holds as well:

\begin{conjecture}
For every $n\geq k$, $\E_n$ has the largest number of copies of $\E_k$ among all binary trees with $n$ leaves.
\end{conjecture}

\section{General results on the inducibility}\label{sec:general}

\begin{theorem}\label{thm:complete}
Complete binary trees contain every fixed binary tree $B$ in a positive proportion: the limit
$$\lim_{h \to \infty} \gamma(B,\CB_h)$$
exists and is strictly positive. In particular, every binary tree $B$ has positive inducibility.
\end{theorem}

\begin{proof}
We prove the statement by induction on the height of $B$ (the greatest distance from the root to a leaf). For height $1$, the statement is trivial: there is only one such tree (namely $B = \C_2$), and in this case we clearly have $\gamma(B,\CB_h) = 1$ for all $h \geq 1$. Now suppose that the statement is true for trees of height at most $H$, and consider a tree $B$ of height $H+1$. Its two branches $B_1$ and $B_2$ clearly both have height at most $H$, so we can apply the induction hypothesis to them. Moreover, we have
\begin{equation}\label{eq:cbh}
c(B,\CB_h) = 2c(B,\CB_{h-1}) + 2 c(B_1,\CB_{h-1}) c(B_2,\CB_{h-1})
\end{equation}
if $B_1$ and $B_2$ are not identical (isomorphic), and
$$c(B,\CB_h) = 2c(B,\CB_{h-1}) + c(B_1,\CB_{h-1}) c(B_2,\CB_{h-1})$$
if they are. This is obtained by distinguishing two possible cases in the same way as we did earlier: the leaves inducing $B$ are either completely contained in one of the branches, or they split in such a way that one part induces $B_1$ in one branch and the rest induces $B_2$ in the other.
Assume that $B_1$ and $B_2$ are distinct, the other case being analogous. Next let $k,k_1,k_2$ be the number of leaves of $B, B_1,B_2$ respectively ($k = k_1 + k_2$). We get
\begin{align*}
\gamma(B,\CB_h) &= \frac{c(B,\CB_h)}{\binom{2^h}{k}} = \frac{2c(B,\CB_{h-1})}{\binom{2^h}{k}} + \frac{2 c(B_1,\CB_{h-1}) c(B_2,\CB_{h-1})}{\binom{2^h}{k}} \\
&= \frac{2\binom{2^{h-1}}{k}}{\binom{2^h}{k}} \cdot \frac{c(B,\CB_{h-1})}{\binom{2^{h-1}}{k}} + \frac{2\binom{2^{h-1}}{k_1}\binom{2^{h-1}}{k_2}}{\binom{2^{h}}{k}} \cdot \frac{c(B_1,\CB_{h-1})}{\binom{2^{h-1}}{k_1}} \cdot \frac{c(B_2,\CB_{h-1})}{\binom{2^{h-1}}{k_2}} \\
&= \frac{2\binom{2^{h-1}}{k}}{\binom{2^h}{k}} \cdot \gamma(B,\CB_{h-1}) + \frac{2\binom{2^{h-1}}{k_1}\binom{2^{h-1}}{k_2}}{\binom{2^{h}}{k}} \cdot \gamma(B_1,\CB_{h-1})  \gamma(B_2,\CB_{h-1}).
\end{align*}
Next we take $\liminf$ and $\limsup$ as $h \to \infty$ (as we did in the proof of Theorem~\ref{thm:even}) and apply the induction hypothesis to $B_1$ and $B_2$. Solving the resulting equations gives us
$$\liminf_{h \to \infty} \gamma(B,\CB_h) = \limsup_{h \to \infty} \gamma(B,\CB_h) = \frac{\binom{k}{k_1}}{2^{k-1}-1} \Big( \lim_{h \to \infty} \gamma(B_1,\CB_{h}) \Big) \Big( \lim_{h \to \infty} \gamma(B_2,\CB_{h})\Big).$$
If $B_1$ and $B_2$ are isomorphic, then we obtain 
$$\liminf_{h \to \infty} \gamma(B,\CB_h) = \limsup_{h \to \infty} \gamma(B,\CB_h) = \frac{\binom{k}{k/2}}{2^{k}-2} \Big( \lim_{h \to \infty} \gamma(B_1,\CB_{h}) \Big)^2$$
in the same way. This completes the proof.
\end{proof}

\begin{theorem}\label{thm:cater_pos}
Caterpillars always form a positive proportion of all trees induced by $k$ leaves in the limit. To be precise,
$$\liminf_{|T| \to \infty} \gamma(\C_k,T) = \frac{k!}{2} \prod_{j=1}^{k-1} (2^j-1)^{-1}.$$
In particular, the only binary trees with inducibility $1$ are caterpillars.
\end{theorem}

This is obtained along the same lines as Proposition~\ref{prop:cb2} and Theorem~\ref{thm:even}. In fact, in the special case $k=4$ they are equivalent: since there are only two different binary trees with four leaves (the caterpillar $\C_4$ and the even tree $\E_4 = \CB_2$), we have
$$\liminf_{|T| \to \infty} \gamma(\C_4,T) = 1 - \limsup_{|T| \to \infty} \gamma(\E_4,T).$$
Before we proceed to the proof, let us first prove a technical lemma similar to Lemma~\ref{lem:f_max1} that will be required in the proof of Theorem~\ref{thm:cater_pos}.

\begin{lemma}\label{lem:f_max}
For every integer $k\geq 2$, the function
$$f(x) = \frac{x(1-x)\big(x^{k-2}+(1-x)^{k-2}\big)}{1-x^k-(1-x)^k}$$
on the interval $(0,1)$ has its minimum at $x = \frac12$, i.e.
$$f(x) \geq f \Big( \frac12 \Big) = \frac{1}{2^{k-1}-1}$$
for all $x \in (0,1)$.
\end{lemma}

\begin{proof}
By the binomial theorem, we have
$$\frac{1}{f(x)} = \frac{\sum_{j=1}^{k-1} \binom{k}{j} x^j(1-x)^{k-j}}{x(1-x)^{k-1} + (1-x)x^{k-1}}.$$
Now we group the terms pairwise as in the proof of Lemma~\ref{lem:f_max1}: for $j\in \{1,2,\ldots,k-1\}$, we have
$$x^j(1-x)^{k-j} + x^{k-j}(1-x)^{j} \leq x(1-x)^{k-1} + (1-x)x^{k-1}$$
for all $x \in (0,1)$, since this is easily seen to be equivalent to
$$\big(x^{j-1}-(1-x)^{j-1}\big)\big(x^{k-j-1}-(1-x)^{k-j-1}\big) \geq 0.$$
It follows that
$$\frac{1}{f(x)} \leq \frac12 \sum_{j=1}^{k-1} \binom{k}{j} = 2^{k-1}-1,$$
which proves the lemma.
\end{proof}

\begin{proof}[Proof of Theorem~\ref{thm:cater_pos}]
 Set $a_k = 
\frac{1}{2} \prod_{j=1}^{k-1} (2^j-1)^{-1}$, and let us prove that for every  $k \geq 3$, there exists a constant $b_k$ such that
$$c(\C_k,T) \geq a_k n^k - b_k n^{k-1}$$
for every binary tree $T$ with $n$ leaves.

This is done by simultaneous induction on $k$ and $n$. For $k = 3$, the statement is trivial since $c(\C_k,T) = \binom{n}{3}$, and so it is for $n=1$. For the induction step, consider the two branches $T_1$ and $T_2$ of $T$. An induced $\C_k$ either consists of leaves in only one of the branches, or it has exactly one leaf in either $T_1$ or $T_2$. The leaves in the other branch have to induce a caterpillar $\C_{k-1}$ in the latter case. Suppose that $T_1$ and $T_2$ have $\alpha n$ and $(1-\alpha)n$ leaves respectively ($0 < \alpha < 1$).
\begin{equation}\label{eq:rec}
c(\C_k,T) = c(\C_k,T_1) + c(\C_k,T_2) + \alpha n c(\C_{k-1},T_2) + (1-\alpha)n  c(\C_{k-1},T_1).
\end{equation}
Using the induction hypothesis, we obtain
\begin{align*}
c(\C_k,T) &\geq a_k (\alpha n)^k - b_k (\alpha n)^{k-1} + a_k \big((1-\alpha) n\big)^k - b_k \big((1-\alpha) n\big)^{k-1} \\
&\quad + \alpha n \Big( a_{k-1} \big((1-\alpha)n\big)^{k-1} - b_{k-1} \big((1-\alpha)n\big)^{k-2} \Big) \\
&\quad + (1-\alpha) n  \Big( a_{k-1} (\alpha n)^{k-1} - b_{k-1} (\alpha n)^{k-2} \Big) \\
&= \Big( a_k\big(\alpha^k + (1-\alpha)^k\big) + a_{k-1}\big(\alpha(1-\alpha)^{k-1}+(1-\alpha)\alpha^{k-1}\big) \Big) n^k \\
&\quad - \Big( b_k\big(\alpha^{k-1} + (1-\alpha)^{k-1}\big) + b_{k-1}\big(\alpha(1-\alpha)^{k-2}+(1-\alpha)\alpha^{k-2}\big) \Big) n^{k-1}.
\end{align*}
Since the function
$$f(x) = \frac{x(1-x)\big(x^{k-2}+(1-x)^{k-2}\big)}{1-x^k-(1-x)^k}$$
has its minimum at $x = \frac12$ by Lemma~\ref{lem:f_max}, and the value there is $f(\frac12) = \frac{1}{2^{k-1}-1}$, we have
$$a_k = \frac{1}{2^{k-1}-1} a_{k-1} \leq \frac{\alpha(1-\alpha)\big(\alpha^{k-2}+(1-\alpha)^{k-2}\big)}{1-\alpha^k-(1-\alpha)^k}a_{k-1}$$
and consequently
$$\Big( a_k\big(\alpha^k + (1-\alpha)^k\big) + a_{k-1}\big(\alpha(1-\alpha)^{k-1}+(1-\alpha)\alpha^{k-1}\big) \Big) n^k \geq a_k n^k.$$
We choose the constant $b_k$ is in such a way that
$$b_k \geq b_{k-1} \cdot \sup_{0 < x  < 1}  \frac{x(1-x)\big(x^{k-3}+(1-x)^{k-3}\big)}{1-x^{k-1}-(1-x)^{k-1}}.$$
The supremum is actually finite (and independent of $n$, which is crucial!), since the limits at $0$ and $1$ both are. It then follows that
$$\Big( b_k\big(\alpha^{k-1} + (1-\alpha)^{k-1}\big) + b_{k-1}\big(\alpha(1-\alpha)^{k-2}+(1-\alpha)\alpha^{k-2}\big) \Big) n^{k-1} \leq b_k n^{k-1},$$
which completes the induction step. Thus we end up with
$$\liminf_{|T| \to \infty} \gamma(\C_k,T) \geq \lim_{n \to \infty} \frac{a_kn^k-b_kn^{k-1}}{\binom{n}{k}} = k!a_k = \frac{k!}{2} \prod_{j=1}^{k-1} (2^j-1)^{-1}.$$

\medskip

To complete the proof, we also need to show that this limit can actually be attained. To this end, one can consider complete binary trees. From the recursion
$$c(\C_k,\CB_h) = 2c(\C_k,\CB_{h-1}) + 2^h c(\C_{k-1},\CB_{h-1})$$
that follows from~\eqref{eq:cbh} upon specialisation (see also the proof of Theorem~\ref{thm:complete}), we obtain
$$c(\C_k,\CB_h) = 2^{h-1} \prod _{j=1}^{k-1} \frac{2^h-2^{j-1}}{2^j-1}$$
by another simple induction, starting from the base case $h=1$. Hence we have
$$\lim_{h \to \infty} \gamma(\C_k,\CB_h) = \lim_{h \to \infty} \frac{2^{h-1}}{\binom{2^h}{k}}\prod _{j=1}^{k-1} \frac{2^h-2^{j-1}}{2^j-1} = \frac{k!}{2} \prod_{j=1}^{k-1} (2^j-1)^{-1},$$
completing our proof.
\end{proof}

We conclude this section with two open problems. First, we notice that the inducibility is a rational number in all our examples. It is thus natural to ask whether this is always the case.

\begin{question}
Is the inducibility of a binary tree always a rational number?
\end{question}

As a second point, it seems (experimentally) that the maximum of $\gamma(B,T)$ over trees of given order $n$ is usually close to the eventual limit superior $i(B)$. Specifically, we conjecture the following:

\begin{conjecture}
For every binary tree $B$, we have
$$\max_{|T| = n} \gamma(B,T) = i(B) + O(n^{-1}).$$
\end{conjecture}

\section{An application: crossing numbers of  random tanglegrams} \label{tangle}

A tanglegram consists of two binary trees with the same number of leaves and a perfect matching between the leaves of the two trees. Tanglegrams play a major role in phylogenetics, especially in the theory of cospeciation. The first binary tree is the phylogenetic tree of hosts, while the second binary tree is the phylogenetic tree of their parasites, e.g. gopher and louse 
\cite{HafnerNadler}.

There are generally many different ways to draw a given tanglegram (see Figure~\ref{fig:two_drawings} for two different drawings of the same tanglegram). Note that edges that belong to the perfect matchings are allowed to cross, while this is not permitted for tree edges. Is is desirable to draw a tanglegram with the least possible number of crossings, which is known as the Tanglegram Layout Problem \cite{St.John}.  The least possible number of crossings  is related to the  number of times  parasites switched hosts \cite{HafnerNadler},
or, working with gene trees instead of phylogenetic trees, to the 
number of horizontal gene transfers (\cite{Burt}, pp. 204--206).
Since the  Tanglegram Layout Problem  is known to be NP-hard, information about the order of magnitude of the minimum number of crossings can be useful.

In the following, we show that a random tanglegram of size $n$ (i.e., both trees have $n$ leaves) has a crossing number of order $\Theta(n^2)$, both on average and with high probability. Since it is trivial that the number of crossings of any tanglegram drawing is at most $\binom{n}{2}$, this shows that most tanglegrams actually require ``many'' crossings.
The proof uses the ``counting method"  of crossing number theory \cite{Zarank}, based on the fact that many induced tanglegrams of order 4 must have a crossing.
We do not shoot for the best constant as constants in the most basic crossing number problems are not known either \cite{Zarank}.

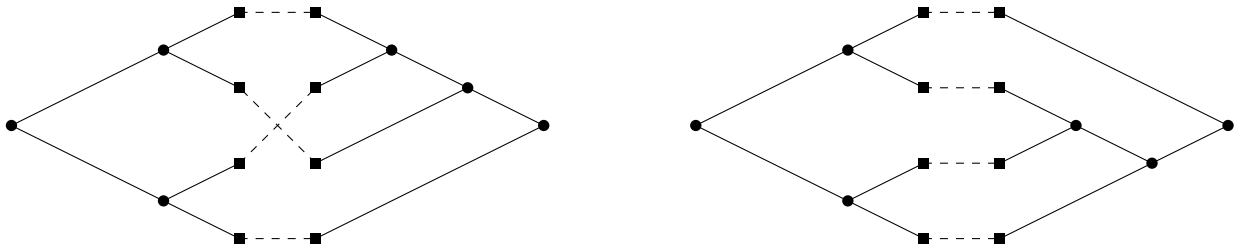
\begin{figure}[htbp]
\begin{center}
\begin{tikzpicture}
        \node[fill=black,circle,inner sep=1.5pt]  at (-1,0) {};
        \node[fill=black,circle,inner sep=1.5pt]  at (1,-1) {};
        \node[fill=black,circle,inner sep=1.5pt]  at (1,1) {};
        \node[fill=black,rectangle,inner sep=2pt]  at (2,-1.5) {};
        \node[fill=black,rectangle,inner sep=2pt]  at (2,-.5) {};
        \node[fill=black,rectangle,inner sep=2pt]  at (2,.5) {};
        \node[fill=black,rectangle,inner sep=2pt]  at (2,1.5) {};

	\draw (-1,0)--(2,1.5);
	\draw (-1,0)--(2,-1.5);
	\draw (1,-1)--(2,-.5);
	\draw (1,1)--(2,.5);

        \node[fill=black,rectangle,inner sep=2pt]  at (3,-1.5) {};
        \node[fill=black,rectangle,inner sep=2pt]  at (3,-.5) {};
        \node[fill=black,rectangle,inner sep=2pt]  at (3,.5) {};
        \node[fill=black,rectangle,inner sep=2pt]  at (3,1.5) {};
        \node[fill=black,circle,inner sep=1.5pt]  at (4,1) {};
        \node[fill=black,circle,inner sep=1.5pt]  at (5,.5) {};
        \node[fill=black,circle,inner sep=1.5pt]  at (6,0) {};

	\draw (6,0)--(3,1.5);
	\draw (6,0)--(3,-1.5);
	\draw (5,.5)--(3,-.5);
	\draw (4,1)--(3,.5);

	\draw [dashed] (2,-1.5)--(3,-1.5);
	\draw [dashed] (2,-.5)--(3,.5);
	\draw [dashed] (3,-.5)--(2,.5);
	\draw [dashed] (2,1.5)--(3,1.5);

        \node[fill=black,circle,inner sep=1.5pt]  at (8,0) {};
        \node[fill=black,circle,inner sep=1.5pt]  at (10,-1) {};
        \node[fill=black,circle,inner sep=1.5pt]  at (10,1) {};
        \node[fill=black,rectangle,inner sep=2pt]  at (11,-1.5) {};
        \node[fill=black,rectangle,inner sep=2pt]  at (11,-.5) {};
        \node[fill=black,rectangle,inner sep=2pt]  at (11,.5) {};
        \node[fill=black,rectangle,inner sep=2pt]  at (11,1.5) {};

	\draw (8,0)--(11,1.5);
	\draw (8,0)--(11,-1.5);
	\draw (10,-1)--(11,-.5);
	\draw (10,1)--(11,.5);

        \node[fill=black,rectangle,inner sep=2pt]  at (12,-1.5) {};
        \node[fill=black,rectangle,inner sep=2pt]  at (12,-.5) {};
        \node[fill=black,rectangle,inner sep=2pt]  at (12,.5) {};
        \node[fill=black,rectangle,inner sep=2pt]  at (12,1.5) {};
        \node[fill=black,circle,inner sep=1.5pt]  at (13,0) {};
        \node[fill=black,circle,inner sep=1.5pt]  at (14,-.5) {};
        \node[fill=black,circle,inner sep=1.5pt]  at (15,0) {};

	\draw (15,0)--(12,1.5);
	\draw (15,0)--(12,-1.5);
	\draw (14,-.5)--(12,.5);
	\draw (13,0)--(12,-.5);

	\draw [dashed] (11,-1.5)--(12,-1.5);
	\draw [dashed] (11,-.5)--(12,-.5);
	\draw [dashed] (11,.5)--(12,.5);
	\draw [dashed] (11,1.5)--(12,1.5);

\end{tikzpicture}
\end{center}
\caption{Two layouts of a tanglegram.}\label{fig:two_drawings}
\end{figure}

Billey, Konvalinka, and Matsen \cite{billey} enumerated tanglegrams, and  we use their asymptotic formula for the  number of tanglegrams. They also asked a number of questions about the shape of random tanglegrams that were answered in \cite{KonvWag} by means of a strong structure theorem.

Let us give some formal definitions. A {\em plane binary tree} has one distinguished vertex
assumed to be a common ancestor of all other vertices, and each vertex either has two children
(left and right) or no children. A vertex with no children is a  {\em leaf}, and a vertex with two
children is an {\em internal vertex}. It is well known that the number of plane binary trees with $n$
leaves is the Catalan number $C_n=\frac{1}{n} {2n-2\choose n-1}$. 

An automorphism of a plane binary tree $B$ consists of switching the left and right subtrees in some internal vertices in such a way that $B$ does not change.
The automorphisms of a  plane binary tree $B$ form a group with respect to composition,  the  automorphism group of $B$, denoted by $A(B)$. Given a left plane binary tree $B_1$ and  a right plane binary tree $B_2$, both with $n$ leaves, such that the leaf sets are on two parallel vertical lines,
and a matching  $\sigma$ between their leaf sets drawn in straight line segments
(see Fig. \ref{fig:two_drawings}) is a {\em tanglegram layout}. The crossing number of a layout is the number of crossing pairs of matching edges.
Two layouts represent the same tanglegram if a sequence of switches in internal vertices, and after that a continuous deformation, can move one into the other. (The edges of the matching move with the leaves that they connect. Switching the left and right tree is not allowed.)
The {\em (tangle) crossing number}  $\Crt(T)$ of a tanglegram $T$  is the minimum number of crossings among its layouts.
Let ${\T}_n$ be the set of all tanglegrams of size $n$, and let $t_n$ be the number of elements in the
set  $\T_n$.  Assume that $(B_1,B_2, \sigma)$ is a layout of the tanglegram $T$.
Let $A(T)$ denote the automorphism group of the tanglegram $T$. $A(T)$ can be viewed as either a subgroup of $A(B_1)$ or a subgroup of $A(B_2)$, and we also
write $A(T)=A(B_1,B_2,\sigma)$. See \cite{Matsen} for more information on automorphisms of tanglegrams.

\begin{figure}[h!]
\begin{center}
\begin{tikzpicture}[scale = 0.4]
\newcommand{\treeb}[3]{\coordinate (v1) at (#1,#3); \coordinate (v2) at (#1+#2,#3+0.5);\coordinate (v3) at (#1+2*#2,#3+1);\coordinate (v4) at (#1+3*#2,#3+1.5);\coordinate (v5) at (#1+3*#2,#3+0.5);\coordinate (v6) at (#1+3*#2,#3-0.5);\coordinate (v7) at (#1+3*#2,#3-1.5);\draw[fill] (v1) circle (.5ex);\draw[fill] (v4) circle (.5ex);\draw[fill] (v5) circle (.5ex);\draw[fill] (v6) circle (.5ex);\draw[fill] (v7) circle (.5ex);\draw (v1) -- (v4);\draw (v3) -- (v5);\draw (v2) -- (v6);\draw (v1) -- (v7);}
\newcommand{\treec}[3]{\coordinate (v1) at (#1,#3); \coordinate (v2) at (#1+2*#2,#3-1);\coordinate (v3) at (#1+2*#2,#3+1);\coordinate (v4) at (#1+3*#2,#3+1.5);\coordinate (v5) at (#1+3*#2,#3+0.5);\coordinate (v6) at (#1+3*#2,#3-0.5);\coordinate (v7) at (#1+3*#2,#3-1.5);\draw[fill] (v1) circle (.5ex);\draw[fill] (v4) circle (.5ex);\draw[fill] (v5) circle (.5ex);\draw[fill] (v6) circle (.5ex);\draw[fill] (v7) circle (.5ex);\draw (v1) -- (v4);\draw (v3) -- (v5);\draw (v2) -- (v6);\draw (v1) -- (v7);}
\newcommand{\tangleb}[7]{\treeb{#1}{#2}{#7} \treeb{#1+8*#2}{-#2}{#7} \draw[dashed] (#1 + 3*#2,#7+1.5) -- (#1 + 5*#2,#7+2.5-#3); \draw[dashed] (#1 + 3*#2,#7+0.5) -- (#1 + 5*#2,#7+2.5-#4); \draw[dashed] (#1 + 3*#2,#7-0.5) -- (#1 + 5*#2,#7+2.5-#5); \draw[dashed] (#1 + 3*#2,#7-1.5) -- (#1 + 5*#2,#7+2.5-#6);}
\newcommand{\tanglec}[7]{\treeb{#1}{#2}{#7} \treec{#1+8*#2}{-#2}{#7} \draw[dashed] (#1 + 3*#2,#7+1.5) -- (#1 + 5*#2,#7+2.5-#3); \draw[dashed] (#1 + 3*#2,#7+0.5) -- (#1 + 5*#2,#7+2.5-#4); \draw[dashed] (#1 + 3*#2,#7-0.5) -- (#1 + 5*#2,#7+2.5-#5); \draw[dashed] (#1 + 3*#2,#7-1.5) -- (#1 + 5*#2,#7+2.5-#6);}
\newcommand{\tangled}[7]{\treec{#1}{#2}{#7} \treeb{#1+8*#2}{-#2}{#7} \draw[dashed] (#1 + 3*#2,#7+1.5) -- (#1 + 5*#2,#7+2.5-#3); \draw[dashed] (#1 + 3*#2,#7+0.5) -- (#1 + 5*#2,#7+2.5-#4); \draw[dashed] (#1 + 3*#2,#7-0.5) -- (#1 + 5*#2,#7+2.5-#5); \draw[dashed] (#1 + 3*#2,#7-1.5) -- (#1 + 5*#2,#7+2.5-#6);}
\newcommand{\tanglee}[7]{\treec{#1}{#2}{#7} \treec{#1+8*#2}{-#2}{#7} \draw[dashed] (#1 + 3*#2,#7+1.5) -- (#1 + 5*#2,#7+2.5-#3); \draw[dashed] (#1 + 3*#2,#7+0.5) -- (#1 + 5*#2,#7+2.5-#4); \draw[dashed] (#1 + 3*#2,#7-0.5) -- (#1 + 5*#2,#7+2.5-#5); \draw[dashed] (#1 + 3*#2,#7-1.5) -- (#1 + 5*#2,#7+2.5-#6);}
\tangleb {0} {0.7} 1 2 3 4 {0}
\tangleb {7} {0.7} 1 2 4 3 {0}
\tangleb {14} {0.7} 1 3 2 4 {0}
\tangleb {21} {0.7} 1 3 4 2 {0}
\tangleb {28} {0.7} 1 4 2 3 {0}
\tangleb {0} {0.7} 1 4 3 2 {-5}
\tangleb {7} {0.7} 3 4 1 2 {-5}
\tanglec {14} {0.7} 1 2 3 4 {-5}
\tanglec {21} {0.7} 1 3 2 4 {-5}
\tangled {28} {0.7} 1 2 3 4 {-5}
\tangled {7} {0.7} 1 3 2 4 {-10}
\tanglee {14} {0.7} 1 2 3 4 {-10}
\tanglee {21} {0.7} 1 3 2 4 {-10}
\node at (3,-2.5){$_{\text{No.\ } 1}$};
\node at (10,-2.5){$_{\text{No.\ } 2}$};
\node at (17,-2.5){$_{\text{No.\ } 3}$};
\node at (24,-2.5){$_{\text{No.\ } 4}$};
\node at (31,-2.5){$_{\text{No.\ } 5}$};
\node at (3,-7.5){$_{\text{No.\ } 6}$};
\node at (10,-7.5){$_{\text{No.\ } 7}$};
\node at (17,-7.5){$_{\text{No.\ } 8}$};
\node at (24,-7.5){$_{\text{No.\ } 9}$};
\node at (31,-7.5){$_{\text{No.\ } 10}$};
\node at (10,-12.5){$_{\text{No.\ } 11}$};
\node at (17,-12.5){$_{\text{No.\ } 12}$};
\node at (24,-12.5){$_{\text{No.\ } 13}$};
\end{tikzpicture}
\end{center}
\caption{The $13$ tanglegrams of size $4$ from \cite{KonvWag}.}\label{fig:t4}
\end{figure}
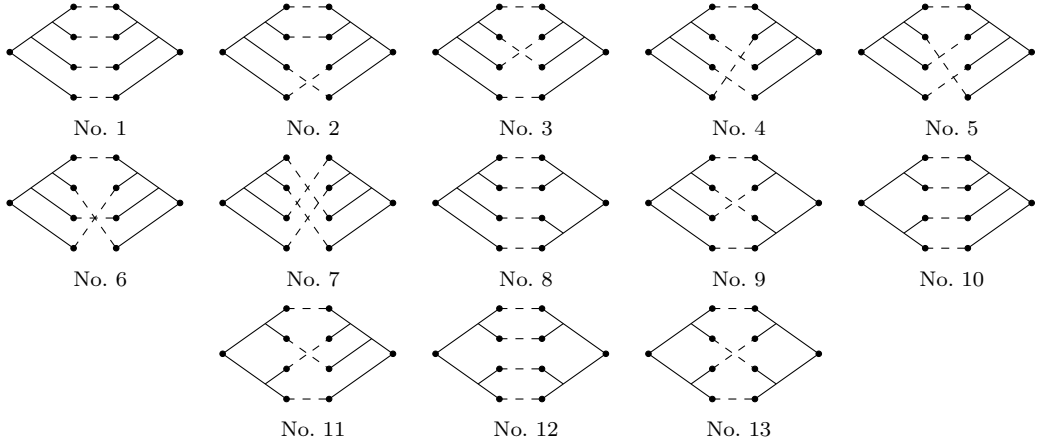

Note that the (ordinary) crossing number $\Cr(T)$ of the graph $T$ can be smaller then $\Crt(T)$, since for the ordinary crossing number we can draw the vertices in any positions, and the edges can be drawn curves.

\begin{lemma}\label{lem:crt} Let $T=(B_1,B_2,\sigma)$ be a tanglegram, where the root vertex of $B_i$ is $r_i$. Then $\Crt(T)\ge \Cr(T^*)$, where $T^*$ is the graph obtained from $T$ by adding the edge $r_1r_2$. 
In particular, if $T_1$ and $T_2$ are the tanglegrams No. $6$ and $13$ in Figure \ref{fig:t4}, then $\Cr(T_1)=\Cr(T_2)=0$ while
$\Crt(T_1)=\Crt(T_2)=1$.
\end{lemma}

\begin{proof} Since in any layout of $T$, the vertices $r_1$ and $r_2$ lie on the infinite face of the drawing, they can be connected with an edge without increasing the crossing number of the drawing, showing $\Cr(T+r_1r_2)\le\Crt(T)$.

Both $T_1$ and $T_2$ are subdivisions of $K_4$ and therefore are planar graphs.
Figure \ref{fig:K3,3} shows a layout with one crossing for the tanglegrams $T_1$ and $T_2$ and a subdivision of $K_{3,3}$ in the graphs
$T_1^*$ and $T_2^*$, showing $\Crt(T_1)=\Crt(T_2)=1$.

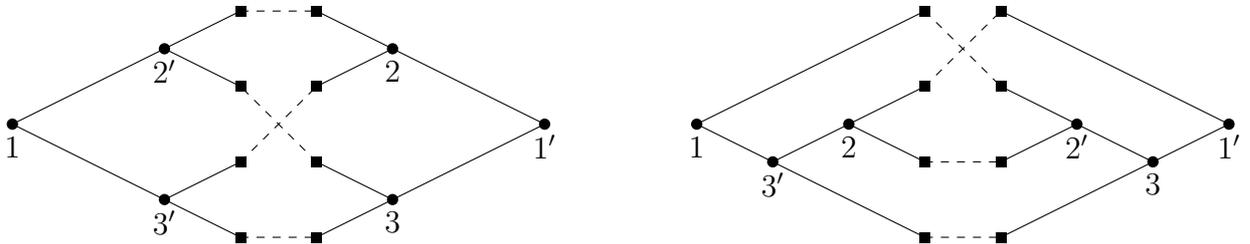
\begin{figure}[htbp]
\begin{center}
\begin{tikzpicture}
  \node[fill=black,circle,inner sep=1.5pt] at (-1,0) {};
        \node[fill=black,circle,inner sep=1.5pt]  at (1,-1) {};
        \node[fill=black,circle,inner sep=1.5pt]  at (1,1) {};
        \node[fill=black,rectangle,inner sep=2pt]  at (2,-1.5) {};
        \node[fill=black,rectangle,inner sep=2pt]  at (2,-.5) {};
        \node[fill=black,rectangle,inner sep=2pt]  at (2,.5) {};
        \node[fill=black,rectangle,inner sep=2pt]  at (2,1.5) {};

	\draw (-1,0)--(2,1.5);
	\draw (-1,0)--(2,-1.5);
	\draw (1,-1)--(2,-.5);
	\draw (1,1)--(2,.5);

        \node[fill=black,rectangle,inner sep=2pt]  at (3,-1.5) {};
        \node[fill=black,rectangle,inner sep=2pt]  at (3,-.5) {};
        \node[fill=black,rectangle,inner sep=2pt]  at (3,.5) {};
        \node[fill=black,rectangle,inner sep=2pt]  at (3,1.5) {};
        \node[fill=black,circle,inner sep=1.5pt]  at (4,1) {};
        \node[fill=black,circle,inner sep=1.5pt]  at (4,-1) {};
        \node[fill=black,circle,inner sep=1.5pt]  at (6,0) {};

	\draw (6,0)--(3,1.5);
	\draw (6,0)--(3,-1.5);
	\draw (4,-1)--(3,-.5);
	\draw (4,1)--(3,.5);

	\draw [dashed] (2,-1.5)--(3,-1.5);
	\draw [dashed] (2,-.5)--(3,.5);
	\draw [dashed] (3,-.5)--(2,.5);
	\draw [dashed] (2,1.5)--(3,1.5);

        \node[fill=black,circle,inner sep=1.5pt]  at (8,0) {};
       \node[fill=black,circle,inner sep=1.5pt]  at (9,-.5) {};
        \node[fill=black,circle,inner sep=1.5pt]  at (10,0) {};
        \node[fill=black,rectangle,inner sep=2pt]  at (11,-1.5) {};
        \node[fill=black,rectangle,inner sep=2pt]  at (11,-.5) {};
        \node[fill=black,rectangle,inner sep=2pt]  at (11,.5) {};
        \node[fill=black,rectangle,inner sep=2pt]  at (11,1.5) {};

	\draw (8,0)--(11,1.5);
	\draw (8,0)--(11,-1.5);
	\draw (9,-.5)--(10,0);
	\draw (10,0)--(11,.5);
	\draw (10,0)--(11,-.5);

        \node[fill=black,rectangle,inner sep=2pt]  at (12,-1.5) {};
        \node[fill=black,rectangle,inner sep=2pt]  at (12,-.5) {};
        \node[fill=black,rectangle,inner sep=2pt]  at (12,.5) {};
        \node[fill=black,rectangle,inner sep=2pt]  at (12,1.5) {};
        \node[fill=black,circle,inner sep=1.5pt]  at (13,0) {};
        \node[fill=black,circle,inner sep=1.5pt]  at (14,-.5) {};
        \node[fill=black,circle,inner sep=1.5pt]  at (15,0) {};

	\draw (15,0)--(12,1.5);
	\draw (15,0)--(12,-1.5);
	\draw (14,-.5)--(12,.5);
	\draw (13,0)--(12,-.5);

	\draw [dashed] (11,-1.5)--(12,-1.5);
	\draw [dashed] (11,-.5)--(12,-.5);
	\draw [dashed] (11,.5)--(12,1.5);
	\draw [dashed] (11,1.5)--(12,.5);
        
        \node at (-1,-.3) {$1$};
        \node at (1,.7) {$2'$};
        \node at (1,-1.3) {$3'$};
        \node at (4,.7) {$2$};
        \node at (4,-1.3) {$3$};
        \node at (6,-.3) {$1'$};

        \node at (15,-.3) {$1'$};
        \node at (14,-.8) {$3$};
        \node at  (13,-.3) {$2'$};
        \node at  (8,-.3) {$1$};
        \node at (9,-.8) {$3'$};
        \node at (10,-.3) {$2$};

\end{tikzpicture}
\end{center}
\caption{Finding copies of $K_{3,3}$ in tanglegrams No. 13 and No. 6.}\label{fig:K3,3}
\end{figure}
\end{proof}

\begin{remark} It is easy to see from F\'ary's Theorem that $\Crt(T)=0$ if and only if $\Cr(T^*)=0$.
\end{remark}

By the orbit-stabilizer theorem, a tanglegram $T$ has $2^{2n-2}/|A(T)|$ different layouts $(B_1,B_2,\sigma)$. Conversely, if we assign a weight $|A(B_1,B_2,\sigma)|/2^{2n-2}$ to every layout and sum over all possibilities, we obtain the total number of tanglegrams:
\begin{equation}\label{eq:tasymp}
t_n=\sum_{T\in \T_n} 1= \sum_{B_1}\sum_{B_2}\sum_\sigma \frac{|A(B_1,B_2,\sigma)|}{2^{2n-2}}\sim n!\frac{e^{1/8}C_n^2}{2^{2n-2}},
\end{equation}  
where the asymptotic formula for $t_n$ is taken from \cite{billey}. On the other hand,
\begin{equation}\label{eq:tasymplow}
t_n=\sum_{T\in \T_n} 1\geq \sum_{B_1}\sum_{B_2}\sum_\sigma \frac{1}{2^{2n-2}}= n!\frac{C_n^2}{2^{2n-2}} \sim \frac{t_n}{e^{1/8}}. 
\end{equation}

\begin{lemma}\label{lem:Aut}
For all $\epsilon>0$ there is a $K>0$ such that for all $n$
\begin{equation}\label{eq:fewbigAut}
\sum_{T\in \T_n \atop |A(T)|\geq K} 1 <\epsilon t_n.
\end{equation}
\end{lemma}
\begin{proof}
It is well-known---and easy to verify using induction---that the order of the  automorphism group of a binary tree is always a power of 2; hence this is also the case for tanglegrams.
Furthermore, it was shown in \cite{KonvWag} that the distribution of the order of the automorphism group is asymptotically Poisson, namely 
$$\lim_{n\rightarrow \infty} \frac{|\{T\in \T_n: |A(T)|=2^k|\}}{t_n} = \frac{e^{-1/4}}{4^{k}k!}.$$
This immediately proves the lemma.
\end{proof}
\begin{lemma}\label{lem:Var}
There are positive constants $c_1$ and $c_2$ such that 
for large enough $n$, arbitrarily fixed left- and right binary plane trees $B_1$ and $B_2$ with $n$ leaves and a matching $\sigma$ selected uniformly at random between the leaves of $B_1$ and $B_2$, for the random tanglegram 
$T_{\sigma}=(B_1,B_2,\sigma)$, we have
\begin{equation}\label{eq:bigE}
\EE_\sigma[\Crt(T_{\sigma})]\geq c_1n^2,
\end{equation}
and
\begin{equation}\label{eq:bigaa}
\PP_{\sigma}[\Crt(T_{\sigma})\geq c_2 n^2] \ge 1-\frac{1}{\sqrt{n}}.
\end{equation}
\end{lemma}
We postpone the proof of Lemma~\ref{lem:Var} and prove the main result on the crossing number of random tanglegrams.

\begin{theorem}\label{th:cr} There is a constant $c_3$ such that 
for a tanglegram $T\in\T_n$ selected uniformly at random,
\begin{equation}\label{eq:bigaaT}
\PP_{\T_n}[\Crt(T)\geq c_3 n^2] =1-o(1)
\end{equation}
as $n\rightarrow \infty$, and hence
\begin{equation}\label{eq:bigET}
\EE_{\T_n}[\Crt(T)] \geq (1-o(1))c_3 n^2.
\end{equation}
\end{theorem}
\begin{proof}
Clearly (\ref{eq:bigaaT}) implies (\ref{eq:bigET}), so we have to prove (\ref{eq:bigaaT}). Select an arbitrary small $\epsilon>0$ and a corresponding $K$ from Lemma~\ref{lem:Aut}.
From (\ref{eq:tasymplow}) and (\ref{eq:bigaa}) we obtain, for large enough $n$,
\begin{align*}
\sum_{T\in \T_n\atop \Crt(T)< c_2n^2} 1 & \leq \epsilon t_n+ \sum_{B_1}\sum_{B_2}\sum_{\sigma 
:    |A(B_1,B_2,\sigma)\leq K \atop  \Crt(B_1,B_2,\sigma)< c_2n^2 }\frac{|A(B_1,B_2,\sigma)|}{2^{2n-2}}\\
& \leq  \epsilon t_n+K \sum_{B_1}\sum_{B_2}\sum_{\sigma:  \Crt(B_1,B_2,\sigma)< c_2n^2   }  \frac{1}{2^{2n-2}}\\
&\leq \epsilon t_n+\frac{K}{\sqrt{n}}\sum_{B_1}\sum_{B_2}\sum_{\sigma }  \frac{1}{2^{2n-2}}= \epsilon t_n+ O\Big(  \frac{K}{\sqrt{n}} t_n\Big).
\end{align*}
\end{proof}

\begin{proof}[Proof of Lemma~\ref{lem:Var}:] 
Assume that $B_1$ induces $h_1$ and $B_2$ induces $h_2$ caterpillars $\C_4$ on four vertices. 
Theorem~\ref{thm:cater_pos} shows that $h_1,h_2$ are asymptotically at least $\frac47 {n\choose 4}$. Looking at the 13 tanglegrams of size 4 (Fig.~\ref{fig:t4}), we recognize that 
most of them have tangle crossing number 0, except No. 6 and No. 13, which have tangle crossing number 1 by Lemma~\ref{lem:crt}. No. 6 has a matching between two $\C_4$'s, No. 13 has a matching between two $\CB_2$'s. We focus on pairs of $\C_4$'s, as
they already provide what we need.

Let $H_1$ (resp. $H_2$) denote the set of all $4$-element subsets of leaves inducing a $\C_4$ in $B_1$ (resp. in $B_2$).  Clearly $|H_i|=h_i$ for $i=1,2$. Select 
a matching $\sigma $ uniformly at random between the leaves of $B_1$ and $B_2$. For $U\in H_1, V\in H_2$, let $X_{UV}(\sigma)=1/{n-2 \choose 2}$ if $\sigma$ matches the vertices
of $U$ to the vertices of $V$, and the induced tanglegram on 4 vertices is the tanglegram No. 6. Otherwise let $X_{UV}(\sigma)=0$.

Now we observe that for every outcome $\sigma$,
\begin{equation} \label{eq:crlower}
 \Crt(B_1,B_2,\sigma)\geq \sum_{U\in H_1}\sum_{V\in H_2} X_{UV}(\sigma).
\end{equation} 
To see this, consider an optimal layout of tanglegram $(B_1,B_2,\sigma)$. Whenever an induced tanglegram on 4 vertices is tanglegram  No. 6, some two edges of the matching must
cross. This crossing, however, may arise from many other copies of No. 6. Such copies must contain two further edges from the matching, limiting the number of such copies to at most ${n-2 \choose 2}$. This proves~\eqref{eq:crlower}.

Taking the expected value with respect to $\sigma$ on both sides of (\ref{eq:crlower}), we obtain
\begin{equation} \label{eq:crlowerE}
\EE_{\sigma}[ \Crt(B_1,B_2,\sigma)]\geq \sum_{U\in H_1}\sum_{V\in H_2}  \EE_{\sigma}[X_{UV}].
\end{equation}
Observe that $\EE_{\sigma}[X_{UV}]=\frac{4(n-4)!}{{n-2 \choose 2}n!}$: there are four ways to connect the leaves of two $C_4$'s in such a way that a tanglegram isomorphic to No. 6 is obtained, and then $(n-4)!$ ways to connect the remaining leaves. Hence (using the notation $(n)_k = n!/(n-k)!$)
\begin{equation} \label{eq:crlowerEasymp}
\EE_{\sigma}[ \Crt(B_1,B_2,\sigma)]\geq \frac{4h_1h_2}{{n-2 \choose 2}(n)_4}\geq \frac{2-o(1)}{441}n^2,
\end{equation}
where the $o(1)$ terms depends on $n$ only, but not on $B_1$ or $B_2$. Formula (\ref{eq:crlowerEasymp}) verifies the first claim of the lemma, (\ref{eq:bigE}).
 
To prove the second claim, we use the second moment method.  Set $X=\sum_{U\in H_1}\sum_{V\in H_2} X_{UV}$. We estimate
\begin{align} \nonumber
{\rm Var}[X] &= \EE[(X-\EE[X])^2]\\
&=\sum_{U\in H_1}\sum_{V\in H_2} {\rm Var}[X_{UV}]+ \sum_{U\in H_1\atop V\in H_2}   \sum_{W\in H_1\atop {Z\in H_2 \atop \{U,V\}\not= \{W,Z\} }} \Big(
\EE[X_{UV}X_{WZ}]-\EE[X_{UV}]\EE[X_{WZ}]
\Big).   \label{eq:Var}
\end{align}

It is easy to see that $\sum_{U\in H_1}\sum_{V\in H_2} {\rm Var}[X_{UV}]$ is bounded above by an absolute constant. In the second summation,
whenever $|U\cap W|\not= |V\cap Z|$, we have $X_{UV}X_{WZ}=0$, so that the covariance term is less than 0. We are left with the cases
$|U\cap W|= |V\cap Z|=i$, for $i=0,1,2,3$. 

If $|U\cap W|= |V\cap Z|$, then we have 
$$\EE[X_{UV}X_{WZ}] = \frac{16(n-8)!}{{n-2\choose 2}^2n!} = \frac{16}{{n-2\choose 2}^2(n)_8},$$
since there are $4^2 = 16$ ways to connect $U$ and $V$ respectively $W$ and $Z$ in such a way that two copies of tanglegram No. 6 are formed, as well as $(n-8)!$ ways to connect the other leaves. Thus the total contribution of the covariance terms for $i=0$ is bounded above by
$${n \choose 4}^4 \Bigg( \frac{16}{{n-2\choose 2}^2(n)_8}-   \frac{16}{{n-2\choose 2}^2(n)_4^2}   \Bigg) = O(n^3).$$      
Likewise, the covariance terms for $i=1,2,3$ can be estimated by
$$O(n^{16-2i})  \Bigg( \frac{1}{{n-2\choose 2}^2(n)_{8-i}}-   \frac{1}{{n-2\choose 2}^2(n)_4^2}   \Bigg) = O(n^{4-i}).$$  
We conclude that ${\rm Var}[X]=O(n^3)$. Chebyshev's inequality yields
$$\PP_\sigma\bigg[|X-\EE_\sigma[X]|\geq k\sqrt{{\rm Var}[X]}\bigg] \leq \frac{1}{k^2}.$$
Setting $k=n^{1/4}$ means
$$\PP_\sigma\bigg[\Crt(B_1,B_2,\sigma)\geq  \EE_\sigma[X] -n^{1/4}\sqrt{{\rm Var}[X]}
\bigg]\geq \PP_\sigma\bigg[X\geq  \EE_\sigma[X] -n^{1/4}\sqrt{{\rm Var}[X]}\bigg] \geq 1-\frac{1}{\sqrt{n}},$$
proving (\ref{eq:bigaa}).
\end{proof}


\end{document}